\pdfoutput=1
\documentclass{amsart}
\usepackage{graphicx}
\usepackage{amssymb}
\usepackage{amsmath}
\usepackage{mathrsfs}
\usepackage{esint}
\usepackage[shortlabels]{enumitem}
\usepackage{xcolor}
\usepackage{color}
\newtheorem{thm}{Theorem}[section]
\newtheorem{cor}[thm]{Corollary}

\newtheorem{prop}[thm]{Proposition}
\theoremstyle{definition}

\theoremstyle{remark}
\newtheorem{rem}[thm]{Remark}

\begin{document}

\title{Flatness of anisotropic
minimal graphs in $\mathbb{R}^{n+1}$}
\author{Wenkui Du}
\address{Department of Mathematics, University of Toronto}
\email{\tt wenkui.du@mail.utoronto.ca}
\author{Yang Yang}
\address{Department of Mathematics, Johns Hopkins University}
\email{\tt yyang@jhu.edu}

\begin{abstract}
We prove a Bernstein  theorem for $\Phi$-anisotropic  minimal hypersurfaces  in all dimensional Euclidean spaces that the only entire smooth  solutions $u: \mathbb{R}^{n}\rightarrow \mathbb{R}$ of $\Phi$-anisotropic minimal hypersurfaces equation  are affine functions provided the anisotropic area functional integrand $\Phi$ is sufficiently $C^{3}$-close to classical area functional integrand and $|\nabla u(x)|=o(|x|^{\varepsilon})$ for $\varepsilon\leq \varepsilon_{0}(n, \Phi)$ with the constant $\varepsilon_{0}(n, \Phi)>0$.
\end{abstract}

\maketitle

\section{Introduction}

In this paper we study entire graphical  solutions of  anisotropic minimal hypersurfaces in $\mathbb{R}^{n+1}$ as critical points of  parametric elliptic functionals, which assign each oriented smooth hypersurface $\Sigma \subset \mathbb{R}^{n+1}$ the anisotropic area
\begin{equation}\label{PEF}
 \mathcal{A}_{\Phi} (\Sigma) = \int_{\Sigma} \Phi(\nu)\,d\mu_{\Sigma},
\end{equation}
where $\nu$ is a unit normal vector field along $\Sigma$, and $\Phi$ is a uniformly elliptic integrand, namely, a one-homogeneous function on $\mathbb{R}^{n+1}$ that is positive and smooth on $\mathbb{S}^n$, and satisfies in addition that $\{\Phi < 1\}$ is uniformly convex.  In particular, when $\Phi(x) =|x|$, $\mathcal{A}_{\Phi}(\Sigma)$ recovers the area functional  and the associated critical points are classical minimal hypersurfaces.  These functionals have recently attracted  significant attention due to their practical applications and theoretical significance (\cite{CL1}, \cite{DD}, \cite{DDG}, \cite{DM}, \cite{DT}, \cite{DMMN}, \cite{FM}, \cite{FMP}). Specifically, they emerge in the context of crystal surface models and Finsler geometry. They also present some challenging and technical difficulties, particularly due to the absence of a monotonicity formula (see \cite{All}), which often leads to more general and illuminating ideas and proofs even in the classical area functional case.

If $\Sigma$ is the graph of a smooth function $u$ on a domain $\Omega \subset \mathbb{R}^n$, then we can rewrite the integral \eqref{PEF} as
\begin{equation}\label{def2_Aphi}
    \mathcal{A}_{\Phi} (\Sigma) = \int_{\Omega} \varphi (\nabla u) \,dx,
\end{equation}
where 
\begin{equation}
    \varphi (x) : = \Phi (-x, \,1).
\end{equation}
Thus, if $\Sigma$ is a critical point of $\mathcal{A}_{\Phi}$, then $u$ solves the Euler-Lagrange equation
\begin{equation}\label{ELE}
    \mathrm{div}(\nabla \varphi (\nabla u)) = 0
\end{equation} in $\Omega$,  {which is equivalent to that the anisotropic mean curvature later defined in \eqref{vanish_Hphi} vanishes}. The function $\varphi$ exhibits local uniform convexity due to the uniform ellipticity of $\Phi$. However, the ratio of the minimum to maximum eigenvalues of $D^2\varphi$ becomes degenerate at infinity. Therefore, equation \eqref{ELE} can be characterized as a quasilinear degenerate elliptic partial differential equation for $u$, which is well-known in the literature as a variational equation of minimal surface or mean curvature type (see e.g. Chapter 16 in \cite{GT}).

The Bernstein problem asks when the minimal hypersurfaces in $\mathbb{R}^{n+1}$ are flat, which is closely related to the regularity of minimal hypersurfaces. It is well-known that Bernstein theorem says that  the only  entire smooth graphical solutions of the classical graphical minimal hypersurface equation on $\mathbb{R}^n$
are affine functions, provided that $n\leq 7$ (see \cite{A}, \cite{D}, \cite{F}, \cite{S})\footnote{More generally, by the work of Fischer-Colbrie, Schoen, do Carmo, Peng, and Pogorelov (\cite{DCP}, \cite{P}, \cite{FCS}) and the recent breakthrough of Chodosh and Li \cite{CL2}, Bernstein theorem holds for complete connected two sided stable minimal hypersurfaces in $\mathbb{R}^{n+1}$ where $n=2, 3$. For $n=4, 5$, Bernstein theorem for complete connected two sided stable minimal hypersurfaces holds under assuming certain volume growth conditions (see \cite{SSY},\cite{SS}, \cite{S}).},  and the result fails when $n>7$ by the counter example  of Bombieri, De Giorgi and Giusti in \cite{BDG}.  Later, Ecker and Huisken \cite{EH} extended Bernstein theorem to all dimensions, assuming in addition 
\begin{equation}
    |\nabla u(x)| = o(|x|).
\end{equation}

In this paper, we consider Bernstein problem for anisotropic minimal hypersurfaces in $\mathbb{R}^{n+1}$. For anisotropic area functional $\mathcal{A}_{\Phi}$ with uniformly elliptic integrand $\Phi$, by the works of Jenkins \cite{J} and Simon \cite{Si2} all entire smooth $\mathcal{A}_{\Phi}$-minimal graphs are flat when $n = 2, \,3$. A recent breakthrough \cite{MY1, MY2} by Mooney and the second author, and the work \cite{M} of Mooney showed that the anisotropic Bernstein theorem for entire smooth graphical solutions can only hold up to $n=3$ by constructing nonflat minimizers  when $n= 4, 5, 6$ for certain $\Phi$ away from the area integrand; for the case $n=4$, the gradient of solutions they constructed has growth rate between $(0, \,1/2)$ at infinity \cite{MY2}, so the result in \cite{EH} does not hold for general uniformly elliptic anisotropic area functionals. In \cite[Conjecture 3.3]{MY1}, Mooney and the second author conjectured that if the solution $u$ of an anisotropic minimal hypersurface equation satisfies $\sup_{B_r} |\nabla u| = o(r^{\varepsilon})$ for some $\varepsilon$ depending on $n$ and $\Phi$, then the solution $u$ must be an affine function.


The known results for anisotropic minimal hypersurfaces are, for example, the Bernstein theorem still holds up to dimension $n = 7$ by work of Simon \cite{Si2} under small $C^3$ perturbations of the integrand in (\ref{PEF}) from area on $\mathbb{S}^{n}$, and the flatness of stable {anisotropic minimal hypersurfaces} holds in dimension $n = 3$ under $C^4$ perturbations by Chodosh-Li \cite[Corollary 1.5]{CL1}\footnote{Indeed  by \cite[Remark 1.6]{CL1},  their result still holds in weaker $C^{2, \alpha}$ closeness assumption, but the closeness cannot be quantitatively estimated due to using a contradiction argument from \cite{Si2}.}. It would be interesting to determine whether or not the topology in which closeness is measured could be relaxed to $C^3$-closeness sense  in all dimensions such that Bernstein theorem holds. The purpose of this paper is to give an affirmative answer to this question and Conjecture 3.3 in \cite{MY1} for $\mathcal{A}_{\Phi}$-minimal graphs under assuming certain growth condition on $|\nabla u|$ and $C^3$-closeness between $\Phi$ and classical area integrand. In particular  our main result is:\\
\begin{thm} \label{Main Theorem}
Let $u$ be an entire smooth solution to the anisotropic minimal surface  equation 
\begin{equation}
    \mathrm{div}(\nabla \varphi (\nabla u)) = \partial^2_{ij}\varphi(\nabla u)\partial^2_{ij}u=0
\end{equation}
on $\mathbb{R}^n$, corresponding to a uniformly elliptic anisotropic area functional  $\mathcal{A}_{\Phi}$ in \eqref{PEF}.
There is a constant $\delta(n)>0$, such that if
\begin{equation}
    \|\Phi-1\|_{C^3(\mathbb{S}^n)}\leq\delta(n),
\end{equation}
then there is
another constant $\varepsilon_{0}(n, \Phi)>0$, such that for every $\varepsilon \leq \varepsilon_{0}(n, \Phi)$  if
\begin{equation}\label{gradient_growth}
    \sup_{B_r} |\nabla u| = o(r^{\varepsilon}),
\end{equation}
then $u$ is an affine function. Moreover, $\varepsilon_{0}(n, \Phi)$ converges to $1$ as $\|\Phi-1\|_{C^3(\mathbb{S}^n)}\to 0$.\\
\end{thm}

\begin{rem}
    Compared with the work of Chodosh-Li \cite[Corollary 1.5]{CL1}, Winklmann's work \cite{W, W2}, and the work of Schoen-Simon-Yau \cite{SSY}, our result holds for  {entire smooth anisotropic minimal graphs in all dimensions, and only requires $C^{3}$-closeness between the anisotropic integrand $\Phi$ and the classical area functional integrand}. Compared with the $C^{3}$-closeness version result of \cite{Si2} for $n  \leq 7$ or the $C^{2,\alpha}$-closeness version result of \cite[Remark 1.6]{CL1} for $n=3$, {by tracing the proof of Theorem \ref{Main Theorem}} our approach actually allows for a quantitative estimation of the $C^{3}$-closeness between the anisotropic area integrand and the classical area integrand. Our result also extends the work of Ecker and Huisken \cite{EH} to the context of anisotropic minimal hypersurfaces, and it recovers their result by setting $\varepsilon=1$ when $\Phi|_{\mathbb{S}^{n}}=1$. Additionally, it expands upon the work of Simon \cite{Si2} to encompass all dimensions under the gradient growth condition \eqref{gradient_growth}.
\end{rem}
The paper is organized as follows. In Section 2 we explain some notations and recall some well-known equations for anisotropic minimal hypersurfaces which will be used for the remaining of the paper. In Section 3, we will prove an almost mean value inequality on anisotropic minimal hypersurfaces via a generalized version of Simons' inequality with divergence form of nonlinear perturbation terms. In Section 4, we show two  integral estimates appearing in the above almost mean value inequality.  In Section 5, we will use the above ingredients to finish proving our main result Theorem \ref{Main Theorem}.

\section*{Acknowledgments}
W. Du appreciates the support from the NSERC Discovery Grant RGPIN-2019-06912 of Prof. Y. Liokumovich at the University of Toronto. Y. Yang gratefully acknowledges the support of the Johns Hopkins University Provost’s Postdoctoral Fellowship Program. Y. Yang would also like to thank Prof. C. Mooney at the University of California, Irvine for bringing this problem to his attention and for many inspiring and helpful discussions.

We are very grateful to the referee for the helpful comments regarding presentation, and for bringing our attention to the question discussed in Remark \ref{nonlinear}.

\section{Preliminaries}

In this section, we introduce some notations and formulae to be used in subsequent sections.  Let  
$g$ be the metric on smooth hypersurface $\Sigma\subset \mathbb{R}^{n+1}$ induced from $\mathbb{R}^{n+1}$ and $\nabla^2\Phi$ be the Hessian of $\Phi$ over sphere $\mathbb{S}^{n}$. Then,  let $A_{\Phi}=D^2\Phi(\nu)$ be identified as a positive definite  symmetric $(1, 1)$ tensor  defined by 
\begin{equation}
\begin{split}
    g(A_{\Phi}X, Y)&=\nabla^2\Phi(\nu)(d\nu(X), d\nu(Y)) \\
    &=D^2\Phi(\nu)(X, Y)=g(D^2\Phi(\nu)X, Y),\quad X, Y\in T\Sigma,
    \end{split}
\end{equation} 
and $B_{\Phi}=A^{-1}_{\Phi}$, where $\nu$ is a unit normal vector field along $\Sigma$ and the Weingarten map identifies the tangent space of a point on hypersurface  and the tangent space of Gauss map  image point on sphere. We also denote the smallest eigenvalue of $A_{\Phi}$ by $\lambda_{\Phi}>0$,  and from \cite[Lemma 2.1]{W} we have 
\begin{equation}
\lambda_{\Phi} \rightarrow 1 
\end{equation}
as $\lVert \Phi - 1\rVert_{C^2(\mathbb{S}^n)} \rightarrow 0.$ \\
We define the $\Phi$-anisotropic metric on $\Sigma$ by 
\begin{equation}
    (X, Y)_{\Phi}:= \bar{g}_{\Phi}(X, Y)= g(X, B_{\Phi}Y),
\end{equation}
whose local representation is
\begin{equation}
    \bar{g}_{ij}=B_{i}^{\ell}g_{\ell j},
\end{equation}
and $B_{i}^{\ell}$ is the local representation of $B_{\Phi}$.  

The $\Phi$-anisotropic gradient of a function $\phi$ defined on $\Sigma$ is given by
\begin{equation}
    \bar{\nabla}\phi= D^2\Phi(v)\nabla \phi ,
\end{equation}
and 
\begin{equation}
    |\bar{\nabla} \phi|^2= (\bar{\nabla} \phi, \bar{\nabla} \phi)_{\Phi} =g(D^2\Phi(v)\nabla \phi, \nabla\phi).
\end{equation}
More generally, for tensor field $T$ on $\Sigma$ whose local coordinates representation is $T_{j_1...j_s}^{i_1...i_r}$, we can define the $\Phi$-covariant derivatives $\bar{\nabla}$ with respect to the Levi-Civita connection associated with metric $g_{\Phi}$ (or equivalently Christoffel symbol $\bar{\Gamma}_{ij}^k$ related to third order derivative of $\Phi$).
The associated $\Phi$-anisotropic norm $|\cdot|$ is given by
\begin{equation}
    |T|^2 = \bar{g}_{i_1k_1}...\bar{g}_{i_rk_r}\bar{g}^{j_1l_1}...\bar{g}^{j_sl_s}T_{j_1...j_s}^{i_1...i_r}T_{l_1...l_s}^{k_1...k_r}.
\end{equation}

Let $S=-d\nu$ be the shape operator on $\Sigma\subset \mathbb{R}^{n+1}$ with $S^{j}_{i}=g^{ik}h_{kj}$ and $h_{kj}$ being the second fundamental form of $(\Sigma, g)$. The first variation formula of $\mathcal{A}_{\Phi}$ in \cite[A.1]{CL1} shows that $\Sigma$ is $\mathcal{A}_{\Phi}$ minimal hypersurface if and only if the anisotropic mean curvature
\begin{equation}\label{vanish_Hphi}
    H_{\Phi}=tr(S_{\Phi})=tr_{\Sigma}(D^{2}\Phi(\nu) S)=0,
\end{equation}
where $S_{\Phi}=D^{2}\Phi(\nu) S$ is the $\Phi$-anisotropic shape operator on $\Sigma$ with $\bar{S}^{j}_{i}=\bar{g}^{ik}h_{kj}$.
The second variation formula in \cite[A.2]{CL1} gives the anisotropic Jacobi operator of $\mathcal{A}_{\Phi}$
\begin{equation}
    L_{\Phi} = \Delta_{\Phi} + |S_{\Phi}|^2,
\end{equation}
where $\Delta_{\Phi}$ is given by
\begin{equation}
\Delta_{\Phi} \phi =\text{div}(\bar{\nabla}\phi)= \text{div}\left(D^2\Phi(\nu) \nabla\phi\right),
\end{equation}
and
\begin{equation}
    |S_{\Phi}|^2=\textrm{tr}\left(D^2\Phi(\nu) \cdot S^2\right).
\end{equation}


\section{Generalized Simons' inequality}

In this section, we prove a generalized Simons' inequality and an almost mean value inequality for anisotropic minimal graphs which only require the $C^{3}$-closeness between the anisotropic area functional integrand $\Phi$ and classical area functional integrand. 
\begin{prop}[Generalized Simons' inequality]\label{gsi}
There is constant $\delta_{0}(n)>0$  such that for $\|\Phi-1\|_{C^{3}(\mathbb{S}^{n})}<\delta_{0}(n)$ and for $p>0$ large enough (depending on $n$) and $q$ satisfying
\begin{equation}
     \frac{n}{n+1}q\leq p-\frac{n}{n+1},\quad \quad q\geq G(n, \Phi)p
\end{equation}
 we have
\begin{equation}\label{subsolution_Svpq}
\Delta_{\Phi}(|S_{\Phi}|^pv^q) \geq\, \left(q-Gp\right)|S_{\Phi}|^{p+2}v^q+p\bar{\nabla}_{i}{E}^{i}_{(12)},
\end{equation}
where $v= \sqrt{1+|\nabla u|^2}, \,  G=G(n, \Phi)>0$ is a constant depending on  $n$ and integrand $\Phi$ and converging to $1$ as $\|\Phi-1\|_{C^{3}(\mathbb{S}^{n})}\to 0$, and the vector field ${E}_{(12)}$ satisfies
\begin{equation}\label{E12prop}
    |{E}_{(12)}|\leq \varepsilon(\Phi)|S_{\Phi}|^{p+1}v^{q}
\end{equation}
with $0<\varepsilon(\Phi)\ll 1$  converging to $0$ as $\|\Phi-1\|_{C^{3}(\mathbb{S}^{n})}\to 0$.
\end{prop}
\begin{proof}
We first notice that  the $(n+1)$-th component $\nu^{n+1}$ of the unit normal $\nu$ to $\Sigma$ satisfies the following Jacobi equation 
\begin{equation}
\Delta_{\Phi}\nu^{n+1} + |S_{\Phi}|^2\nu^{n+1} = 0.
\end{equation}
Plugging $v = \left(\nu^{n+1}\right)^{-1} = \sqrt{1+|\nabla u|^2} $ into the above equation, we arrive at
\begin{equation}\label{JEA}
    \Delta_{\Phi} v = |S_{\Phi}|^2v + 2 v^{-1}|\bar{\nabla} v|^2.
\end{equation}
Then we recall the generalized Simons' identity for anisotropic minimal hypersurfaces from \cite[(41)]{W}
\begin{align}\label{Simons_identity}
    \frac{1}{2}\Delta_{\Phi}|S_{\Phi}|^2&=|\bar{\nabla} S_{\Phi}|^2-tr(B_{\Phi}S_{\Phi}^2)|S_{\Phi}|^2+\bar{g}^{kl}T_{ijkl}h^{ij},
\end{align}
where  after some computation 
\begin{align}
T_{ijkl}&=\bar{\nabla}_{i}T^{(1)}_{jlk}+\bar{\nabla}_{k}T^{(2)}_{ijl}+T^{(3)}_{ijkl}\,,
\end{align}
\begin{align}
T^{(1)}_{jlk}=C_{jlk}+E_{kj}^sh_{sl}+E_{kl}^sh_{js}\,,
\end{align}
\begin{align}
T^{(2)}_{ijl}
    &=C_{ijl}-E_{ij}^sh_{sl}-E_{il}^sh_{js}\,,
\end{align}
\begin{align}
   T^{(3)}_{ijkl}= &-(E_{kj}^mE_{mi}^s - E_{ij}^mE_{mk}^s )h_{sl}-(E_{kl}^mE_{mi}^s - E_{il}^mE_{mk}^s)h_{js}\\\nonumber
&+E_{ij}^s\bar{\nabla}_{k}h_{sl}- E_{kj}^s\bar{\nabla}_{i}h_{sl}+E_{il}^s\bar{\nabla}_{k}h_{js} -E_{kl}^s \bar{\nabla}_{i}h_{js}\,,
\end{align}
and the tensors $C_{jki}=E^{l}_{jk}h_{li}-E^{l}_{ij}h_{lj}$ and $E^{k}_{ij}=\bar{\Gamma}^{k}_{ij}-\Gamma^{k}_{ij}$ depend on only third order derivatives of $\Phi$ and  by \cite[Thm 3.4, Cor 3.5]{W} they satisfy
\begin{align}
   |E|\leq \varepsilon(\Phi)|S_{\Phi}|\,, \qquad |C|\leq \varepsilon(\Phi)|S_{\Phi}|^2,
\end{align}
where
\begin{equation}\label{0<e<1}
    0<\varepsilon(\Phi)\ll 1
\end{equation}
represents a constant depending on $\Phi$ and it may vary from text to text but converges to $0$ as $\|\Phi-1\|_{C^3(\mathbb{S}^{n})}\to 0$. Hence, the above equations and estimates imply that
\begin{equation}\label{T12}
    |T^{(1)}|+|T^{(2)}|\leq \varepsilon(\Phi)|S_{\Phi}|^2,
\end{equation}
and
\begin{equation}\label{T3}
    |T^{(3)}|\leq \varepsilon(\Phi)^2|S_{\Phi}|^3+\varepsilon(\Phi)|S_{\Phi}||\bar{\nabla}S_{\Phi}|.
\end{equation}
Then by \eqref{JEA}, \eqref{Simons_identity}, and noticing that $\text{tr}(B_{\Phi}S_{\Phi}^2) \leq \lambda_{\Phi}^{-1}|S_{\Phi}|^2$, where $\lambda_{\Phi}>0$ is the smallest eigenvalue of $A_{\Phi}$,  we compute 
\begin{equation}\label{disaster}
\begin{split}
  \Delta_{\Phi}&(|S_{\Phi}|^pv^q) \geq\, \left(q-\lambda_{\Phi}^{-1}p\right)|S_{\Phi}|^{p+2}v^q+ p\left(p-2 \right)|S_{\Phi}|^{p-2}v^q |\bar{\nabla} |S_{\Phi}||^2 \\
  & +q(q+1)|S_{\Phi}|^p v^{q-2}|\bar{\nabla} v|^2 +2pq|S_{\Phi}|^{p-1}v^{q-1}(\bar{\nabla} |S_{\Phi}|, \bar{\nabla} v)_{\Phi}   \\  
  &+ p|S_{\Phi}|^{p-2}v^q |\bar{\nabla} S_{\Phi}|^2  + p|S_{\Phi}|^{p-2}v^q\bar{g}^{kl}T_{ijkl}h^{ij}.
\end{split}
\end{equation}
After some computation, we have
\begin{align}\label{E123}
    &\quad|S_{\Phi}|^{p-2}v^q\bar{g}^{kl}T_{ijkl}h^{ij}=\bar{\nabla}_{i}E^{i}_{(12)}+E_{(3)},
\end{align}
where
\begin{align}
{E}^{i}_{(12)}=|S_{\Phi}|^{p-2}v^q(\bar{g}^{kl}T^{(1)}_{jlk}h^{ij}+\bar{g}^{il}T^{(2)}_{ijl}h^{kj}),
\end{align}
\begin{align}
     E^{(3)}&=|S_{\Phi}|^{p-2}v^q\bar{g}^{kl}T^{(3)}_{ijkl}h^{ij}\\\nonumber
    &-(p-2)|S_{\Phi}|^{p-3}v^q\bar{g}^{kl}T^{(1)}_{jlk}h^{ij}\bar{\nabla}_{i}|S_{\Phi}|-q|S_{\Phi}|^{p-2}v^{q-1}\bar{g}^{kl}T^{(1)}_{jlk}h^{ij}\bar{\nabla}_{i}v \\\nonumber
    &-(p-2)|S_{\Phi}|^{p-3}v^q\bar{g}^{kl}T^{(2)}_{ijl}h^{ij}\bar{\nabla}_{k}|S_{\Phi}|-q|S_{\Phi}|^{p-2}v^{q-1}\bar{g}^{kl}T^{(2)}_{ijl}h^{ij}\bar{\nabla}_{k}v  \\ \nonumber
    & -|S_{\Phi}|^{p-2}v^q\bar{g}^{kl}T^{(1)}_{jlk}\bar{\nabla}_{i}h^{ij}-|S_{\Phi}|^{p-2}v^q\bar{g}^{kl}T^{(2)}_{ijl}\bar{\nabla}_{k}h^{ij}.
\end{align}
In particular, using \eqref{T12}, \eqref{T3}, uniform boundedness of $|\bar{g}^{-1}|$ for $\Phi$ close to classical area integrand, Young's inequality and recalling our convention on $\varepsilon(\Phi)$ after \eqref{0<e<1}, we have 
\begin{equation}\label{E12}
    |E_{(12)}|\leq \varepsilon(\Phi)|S_{\Phi}|^{p+1}v^{q},
\end{equation}
and 
\begin{align}\label{E3_}
|E^{(3)}|&\leq \varepsilon(\Phi)\bigl(\left(1+\varepsilon(\Phi)\right)|S_{\Phi}|^{p+2}v^{q}+|S_{\Phi}|^{p-2}v^{q}|\bar{\nabla}|S_{\Phi}||^2+ |S_{\Phi}|^{p}v^{q-2}|\bar{\nabla}v|^2 \nonumber\\
& +|S_{\Phi}|^{p-2}v^{q}|\bar{\nabla}S_{\Phi}|^2\bigr).
\end{align}

Now we plug \eqref{E123} into \eqref{disaster}, and by \eqref{E3_} we obtain
\begin{equation}\label{intermediate step}
\begin{split}
  \Delta_{\Phi}&(|S_{\Phi}|^pv^q) \geq\, \left(q-\left(\lambda_{\Phi}^{-1}+2\varepsilon(\Phi)\right)p\right)|S_{\Phi}|^{p+2}v^q+ p(1-\varepsilon(\Phi))|S_{\Phi}|^{p-2}v^q |\bar{\nabla} S_{\Phi}|^2 \\
  & +q\left(q+1-\frac{p}{q}\varepsilon(\Phi)\right)|S_{\Phi}|^p v^{q-2}|\bar{\nabla} v|^2 +2pq|S_{\Phi}|^{p-1}v^{q-1}(\bar{\nabla} |S_{\Phi}|, \bar{\nabla} v)_{\Phi}   \\  
  &+ p\left(p-2-\varepsilon (\Phi)\right)|S_{\Phi}|^{p-2}v^q |\bar{\nabla} |S_{\Phi}||^2 +p\bar{\nabla}_{i}E_{(12)}^{i}.
\end{split}
\end{equation}
Then, we recall that \cite[Lemma 4.1]{W} gives the following estimate for all $\theta>0$ (note that \cite[Lemma 4.1]{W} only requires that $\Phi$ converges to classical area integrand in $C^{3}$-norm sense if we trace its proof) 
\begin{equation}\label{Lemma4.1}
    |\Bar{\nabla}S_{\Phi}|^2 \geq \frac{1+2/n}{1+\theta} |\Bar{\nabla}|S_{\Phi}||^2 -C(\theta)\varepsilon(\Phi)|S_{\Phi}|^4,
\end{equation}
 where $C(\theta)<\infty$  is a nonnegative constant depending only on $\theta>0$.\\
Therefore, by \eqref{intermediate step} and \eqref{Lemma4.1} we obtain
\begin{equation}\label{disaster1}
\begin{split}
  \Delta_{\Phi}&(|S_{\Phi}|^pv^q) \geq\, \left(q-\left(\lambda_{\Phi}^{-1}+2\varepsilon(\Phi) +(1-\varepsilon(\Phi))C(\theta)\varepsilon(\Phi)\right)p\right)|S_{\Phi}|^{p+2}v^q   \\
  &+ p\left(p + (1-\varepsilon(\Phi))\left(\frac{1+2/n}{1+\theta}\right)-2-\varepsilon(\Phi) \right)|S_{\Phi}|^{p-2}v^q |\bar{\nabla} |S_{\Phi}||^2 \\
  & +q\left(q+1-\frac{p}{q}\varepsilon(\Phi)\right)|S_{\Phi}|^p v^{q-2}|\bar{\nabla} v|^2 +2pq|S_{\Phi}|^{p-1}v^{q-1}(\bar{\nabla} |S_{\Phi}|,\bar{\nabla} v)_{\Phi}   \\  
&+p\bar{\nabla}_{i}E_{(12)}^{i}.
     \end{split}
\end{equation}
By Young's inequality  for each $\alpha > 0$ we have
\begin{equation}\label{Young}
    2pq|S_{\Phi}|v (\bar{\nabla} |S_{\Phi}|, \bar{\nabla} v)_{\Phi} \geq -pq\left(\alpha v^2|\bar{\nabla} |S_{\Phi}||^2 + \alpha^{-1}|S_{\Phi}|^2|\bar{\nabla} v|^2\right).
\end{equation}
Combining  \eqref{disaster1} with \eqref{Young}, we obtain
\begin{equation}\label{disaster2}
\begin{split}
  &\quad\Delta_{\Phi}(|S_{\Phi}|^pv^q)\\
  &\geq\, \left(q-\left(\lambda_{\Phi}^{-1}+2\varepsilon(\Phi) +(1-\varepsilon(\Phi))C(\theta)\varepsilon(\Phi)\right)p\right)|S_{\Phi}|^{p+2}v^q+p\bar{\nabla}_{i}E_{(12)}^{i}   \\
  &+ p\left(p -\left(2- (1-\varepsilon(\Phi))\left(\frac{1+2/n}{1+\theta}\right)\right)-\varepsilon(\Phi) -\alpha q\right)|S_{\Phi}|^{p-2}v^q |\bar{\nabla} |S_{\Phi}||^2 \\
  & +q\left(q+1-\frac{p}{q}\varepsilon(\Phi) - \alpha^{-1}p\right)|S_{\Phi}|^p v^{q-2}|\bar{\nabla} v|^2 .
\end{split}
\end{equation}
Now, we fix $\theta=\frac{1}{n}$ and let
\begin{equation}\label{G}
G(n, \Phi)=\left(\lambda_{\Phi}^{-1}+2\varepsilon(\Phi) +(1-\varepsilon(\Phi))C(\theta)\varepsilon(\Phi)\right),
\end{equation} 
so the first term of the right hand side of \eqref{disaster2} is non-negative if and only if 
\begin{equation}\label{C1}
     q\geq G(n, \Phi)p.
\end{equation}
We assert that for $p>2$ large enough, we can find $\alpha>0$ such that
\begin{equation}
    p-\left(2- \frac{1+2/n}{1+\theta}\right)-\alpha q > 0,
\end{equation}
and  
\begin{equation}
    q+1-\alpha^{-1}p> 0,
\end{equation}
so that the last two lines of the right hand side of \eqref{disaster2} would be non-negative. \\
\\
Indeed, when $\varepsilon(\Phi)>0$ is sufficiently small, such $\alpha$ exists if and only if
\begin{equation}\label{interval}
    \frac{p}{q+1} < \frac{p-\left(2- \frac{1+2/n}{1+\theta}\right)}{q},
\end{equation}
or equivalently 
\begin{equation}\label{interval3}
\alpha\in \left(\frac{p}{q+1},\, \frac{p-\left(2-\frac{1+2/n}{1+\theta}\right)}{q}\right).
\end{equation}
Combining the constraints \eqref{C1} and \eqref{interval},  we need
\begin{equation}\label{alpha_inequality}
    G(n, \Phi)\left(2- \frac{1+2/n}{1+\theta}\right)p\leq  \left(2- \frac{1+2/n}{1+\theta}\right)q < p-\left(2- \frac{1+2/n}{1+\theta}\right).
\end{equation}
Recalling $\theta=\frac{1}{n}\in (0, \frac{2}{n})$ and noticing $G(n, \Phi)$ converges to $1$  as $\lVert \Phi -1 \rVert_{C^3(\mathbb{S}^{n})} \rightarrow 0$, we have 
\begin{equation}\label{condition}
    G(n, \Phi)\left(2- \frac{1+2/n}{1+\theta}\right) <1.
\end{equation}   
This guarantees \eqref{alpha_inequality} holds, so there exists a required $\alpha>0$ for $p$ large enough and $q$ satisfying
\begin{equation}
     \left(2- \frac{1+2/n}{1+\theta}\right)q < p-\left(2- \frac{1+2/n}{1+\theta}\right),\quad q\geq G(n, \Phi)p.
\end{equation}
This  together with recalling  $\theta=\frac{1}{n}$ and discarding the last two nonnegative terms in \eqref{disaster2} implies \eqref{subsolution_Svpq}. Noticing \eqref{E12} holds, we
 complete the proof of Proposition \ref{gsi}.\\
\end{proof}

\begin{cor}[almost mean value inequality]\label{aMVP}
    For $m>\frac{n}{2}$,  $p$ large enough and $\|\Phi-1\|_{C^{3}(\mathbb{S}^{n})}$ small enough (depending on $n$) and  $R>0$, we have 
\begin{equation}\label{MVP_equality}
        \begin{split}
        |S_{\Phi}|^pv^{Gp}(0)&\leq c(n, m, p,\Phi) R^{-\frac{n}{2}} \left(\int_{\Sigma \cap B_{R}} |S_{\Phi}|^{2p}v^{2Gp} d\mu_{\Sigma}\right )^{1/2}\\
    &+c(n, m, p, \Phi)\varepsilon(\Phi)R^{1-\frac{n}{2m}}\left(\int_{\Sigma \cap B_{R}} |S_{\Phi}|^{2m(p+1)}v^{2mGp} d\mu_{\Sigma}\right )^{1/2m},
        \end{split}
\end{equation}
 where $G=G(n, \Phi)$ is the constant defined in \eqref{G}, and $\varepsilon(\Phi)>0$ is a constant converging to $0$ as $\|\Phi-1\|_{C^{3}(\mathbb{S}^{n})}\to 0$.
\end{cor}
\begin{proof}
Noticing that every anisotropic minimal graph is an element of anisotropic minimal graphical folitation  along last coordinate direction, by standard calibration argument they are always anisotropic area minimizers. Hence,  we have the following  Sobolev inequality for stable anisotropic minimal hypersurfaces from \cite[(3.7)]{CL1}:
    \begin{equation}
\left(\int_{\Sigma}|f|^{\frac{2n}{n-2}}\right)^{\frac{n-2}{n}}\leq C(n, \Phi)\int_{\Sigma}|\nabla f|^2,
    \end{equation}
Then, by Proposition \ref{gsi} for $q=Gp$ with $p$ large enough we have
\begin{equation}\label{subsolution_geqE}
\Delta_{\Phi}(|S_{\Phi}|^pv^{Gp}) \geq p\bar{\nabla}_{i}{E}^{i}_{(12)}.
\end{equation}
Noticing the uniform ellipticity and boundedness of anisotropic area functional $\Phi$ and nonlinear terms estimates \eqref{E12prop} for right hand side of \eqref{subsolution_geqE}, the desired almost mean value inequality \eqref{MVP_equality} follows from the above Sobolev inequality on stable anisotropic minimal hypersurfaces and standard Nash-Moser iteration (see \cite[Thm 8.17]{GT}).
\end{proof}


\section{Integral estimates}

In this section, we estimate the two integrals in the right hand side of  \eqref{MVP_equality} in Corollary \ref{aMVP}. 
\begin{prop}\label{int1}
For $p$ large enough and $\|\Phi-1\|_{C^{3}(\mathbb{S}^n)}$ small enough (depending on $n, p$)
    \begin{equation}
        \int_{\Sigma } |S_{\Phi}|^{2p}v^{2Gp} \eta^{2p} d\mu_{\Sigma}\leq C_1(n, p, \Phi)\int_{\Sigma} v^{2Gp}|\bar{\nabla}\eta|^{2p}d\mu_{\Sigma},
    \end{equation}
where $G=G(n, \Phi)$ is defined in \eqref{G}, $\eta$ is any test function with compact support and $C_{1}(n, p, \Phi)<\infty$ is a constant.
\end{prop}

\begin{proof}
In \eqref{disaster2}, we replace $p$ with $p-1$, and $q$ with $Gp$. Denoting $\left(2- \frac{1+2/n}{1+\theta}\right)$ by $B$ with $\theta=\frac{1}{n}$, then for every fixed $p > \max \left\{3, \frac{1+B}{1-G\times B}\right\}$, we have
\begin{equation}\label{key2}  
\begin{split}
\Delta_{\Phi}\left(|S_{\Phi}|^{p-1}v^{Gp}\right)&\geq G|S_{\Phi}|^{p+1}v^{Gp} +(p-1)\bar{\nabla}_{i}\hat{E}_{(12)}^{i}
 + M|S_{\Phi}|^{p-3}v^{Gp} |\bar{\nabla} |S_{\Phi}||^2 \\
  &+N|S_{\Phi}|^{p-1} v^{Gp-2}|\bar{\nabla} v|^2,
\end{split}
\end{equation}
where 
\begin{equation}
    M=(p-1)\left(p-1 -\left(2- (1-\varepsilon(\Phi))\left(\frac{1+2/n}{1+\theta}\right)\right)-\varepsilon(\Phi) -\alpha Gp\right),
\end{equation}
\begin{equation}
    N=Gp\left(Gp+1-\frac{p-1}{Gp}\varepsilon(\Phi) - \alpha^{-1}(p-1)\right),
\end{equation}
\begin{equation}
\hat{E}_{(12)}^{i} = |S_{\Phi}|^{p-3}v^{Gp}\left(\bar{g}^{kl}T^{(1)}_{jlk}h^{ij}+\bar{g}^{il}T^{(2)}_{ijl}h^{kj}\right).
\end{equation}
Let $\eta$ be a test function with compact support. Multiplying on both sides of \eqref{key2} by $|S_{\Phi}|^{p-1}v^{Gp}\eta^{2p}$ and taking integral, we obtain
\begin{equation}\label{key3}
\begin{split}
    \int_{\Sigma} |S_{\Phi}|^{2p}v^{2Gp}\eta^{2p} &\leq \frac{1}{G} \int_{\Sigma} |S_{\Phi}|^{p-1}v^{Gp}\eta^{2p}\Delta_{\Phi}\left(|S_{\Phi}|^{p-1}v^{Gp}\right) \\
    &-\frac{p-1}{G}\int_{\Sigma}|S_{\Phi}|^{p-1}v^{Gp}\eta^{2p}\bar{\nabla}_{i}\hat{E}_{(12)}^{i} \\
    &- \frac{M}{G}\int_{\Sigma} |S_{\Phi}|^{2(p-2)}v^{2Gp} \eta^{2p} |\bar{\nabla} |S_{\Phi}||^2\\ 
  &-\frac{N}{G}\int_{\Sigma} |S_{\Phi}|^{2(p-1)} v^{2(Gp-1)}\eta^{2p}|\bar{\nabla} v|^2.
\end{split}
\end{equation}
    Denoting $|S_{\Phi}|^{(p-1)}v^{Gp}$ by $f$, then we have
\begin{equation}\label{key4}
    \begin{split}
        \int_{\Sigma} f \eta^{2p}\Delta_{\Phi} f  &= -2p\int_{\Sigma} f\eta^{2p-1}(\bar{\nabla} f, \bar{\nabla} \eta)_{\Phi} 
         - \int_{\Sigma} |\bar{\nabla} f|^2 \eta^{2p} \\
         &\leq p \int_{\Sigma} \left(\beta^{-1} f^2\eta^{2(p-1)}|\Bar{\nabla}\eta|^2 + \beta |\bar{\nabla}f|^2\eta^{2p} \right)  - \int_{\Sigma}    |\bar{\nabla}f|^2\eta^{2p} \\
         &\leq c(p) \int_{\Sigma} f^2\eta^{2(p-1)}|\bar{\nabla} \eta|^2 \\
         & = c(p) \int_{\Sigma} |S_{\Phi}|^{2(p-1)}v^{2Gp}\eta^{2(p-1)}|\bar{\nabla} \eta|^2.
    \end{split}
\end{equation}
Combining (\ref{key3}) and (\ref{key4}), we obtain
\begin{equation}
    \begin{split}
    \int_{\Sigma} |S_{\Phi}|^{2p}v^{2Gp}\eta^{2p} &\leq \,\frac{c(p)}{G} \int_{\Sigma} |S_{\Phi}|^{2(p-1)}v^{2Gp}\eta^{2(p-1)}|\bar{\nabla} \eta|^2 \\
    &- \frac{p-1}{G}\int_{\Sigma}|S_{\Phi}|^{p-1}v^{Gp}\eta^{2p}\bar{\nabla}_{i}\hat{E}_{(12)}^{i} \\
  &- \frac{M}{G}\int_{\Sigma} |S_{\Phi}|^{2(p-2)}v^{2Gp} \eta^{2p} |\bar{\nabla} |S_{\Phi}||^2\\ 
  &-\frac{N}{G}\int_{\Sigma} |S_{\Phi}|^{2(p-1)} v^{2(Gp-1)}\eta^{2p}|\bar{\nabla} v|^2.
    \end{split}
\end{equation}
On the other hand, by \eqref{E12prop} we have 
\begin{equation}
|\hat{E}_{(12)}^i|\leq \varepsilon(\Phi)|S_{\Phi}|^{p}v^{Gp}.
\end{equation}
Using divergence theorem, we have
\begin{equation}
\begin{split}
&\quad -\frac{p-1}{G}\int_{\Sigma}|S_{\Phi}|^{p-1}v^{Gp}\eta^{2p}\bar{\nabla}_{i}\hat{E}_{(12)}^{i} \\
&= \frac{p-1}{G}\int_{\Sigma}\hat{E}_{(12)}^{i}\bar{\nabla}_{i}\left(|S_{\Phi}|^{p-1}v^{Gp}\eta^{2p}\right) \\
& \leq \varepsilon(\Phi)\frac{p-1}{G}\biggl(\int_{\Sigma} |S_{\Phi}|^{2p-2}v^{2Gp}\eta^{2p}|\bar{\nabla}_{i}|S_{\Phi}|| + \int_{\Sigma} |S_{\Phi}|^{2p-1}v^{2Gp-1}\eta^{2p}|\bar{\nabla}_{i}v| \\
&  + \int_{\Sigma} |S_{\Phi}|^{2p-1}v^{2Gp}\eta^{2p-1}|\bar{\nabla}_{i}\eta|\biggr)\\
& \leq \varepsilon(\Phi)\frac{p-1}{G} \biggl( 3\int_{\Sigma} |S_{\Phi}|^{2p}v^{2Gp}\eta^{2p} + \int_{\Sigma} |S_{\Phi}|^{2(p-1)}v^{2Gp}\eta^{2(p-1)}|\bar{\nabla} \eta|^2  \\
& +\int_{\Sigma} |S_{\Phi}|^{2(p-2)}v^{2Gp}\eta^{2p}|\bar{\nabla} |S_{\Phi}||^2 + \int_{\Sigma} |S_{\Phi}|^{2(p-1)}v^{2(Gp-1)}\eta^{2p}|\bar{\nabla} v|^2
\biggr).
\end{split}
\end{equation}
Combining above with $\varepsilon(\Phi)>0$ small enough (depending on $p$), we obtain
\begin{equation}
    \begin{split}
   & \quad \left(1-\frac{3(p-1)}{G}\varepsilon(\Phi)\right)\int_{\Sigma} |S_{\Phi}|^{2p}v^{2Gp}\eta^{2p} \\
   &\leq \left(\frac{c(p)+(p-1)\varepsilon(\Phi)}{G}\right) \int_{\Sigma} |S_{\Phi}|^{2(p-1)}v^{2Gp}\eta^{2(p-1)}|\bar{\nabla} \eta|^2 \\
  &+\left(\frac{(p-1)\varepsilon(\Phi)}{G}- \frac{M}{G}\right)\int_{\Sigma} |S_{\Phi}|^{2(p-2)}v^{2Gp} \eta^{2p} |\bar{\nabla} |S_{\Phi}||^2\\ 
  &+\left(\frac{(p-1)\varepsilon(\Phi)}{G}- \frac{N}{G}\right)\int_{\Sigma} |S_{\Phi}|^{2(p-1)} v^{2(Gp-1)}\eta^{2p}|\bar{\nabla} v|^2\\
  &\leq \left(\frac{c(p)}{G}+\varepsilon(\Phi)\right) \int_{\Sigma} |S_{\Phi}|^{2(p-1)}v^{2Gp}\eta^{2(p-1)}|\bar{\nabla} \eta|^2 .
    \end{split}
\end{equation}
The above inequality together with Young's inequality implies that
\begin{equation}
    \int_{\Sigma} |S_{\Phi}|^{2p}v^{2Gp}\eta^{2p}\leq C_1(n, p, \Phi)\int_{\Sigma} v^{2Gp}|\bar{\nabla}\eta|^{2p}.
\end{equation}
\end{proof}
Then, we estimate the following integral. \\
\begin{prop}\label{int2}
For $m>\frac{n}{2}$, $p$ large enough  and $\|\Phi-1\|_{C^{3}(\mathbb{S}^n)}$ small enough (depending on $m, n, p$), we have
    \begin{equation}
       \int_{\Sigma} |S_{\Phi}|^{2m(p+1)}v^{2mGp}\eta^{2m(p+1)} d\mu_{\Sigma}\leq C_2(n, p, \Phi)\int_{\Sigma} v^{2mGp}|\bar{\nabla}\eta|^{2m(p+1)}d\mu_{\Sigma},
    \end{equation}
    where $G=G(n, \Phi)$ is defined in \eqref{G}, $\eta$ is any test function with compact support and $C_{2} (n, p, \Phi)<\infty$ is a constant.
\end{prop}

\begin{proof}
 The proof of Proposition \ref{int2} would be  similar to that of Proposition \ref{int1}, but we have to more carefully choose our parameters (e.g. $m, \,p$, and $G$) to make the new terms in the estimates can be absorbed into the desired ones.  
 
 In \eqref{disaster2}, we replace $p$ with $m(p-1)$, and $q$ with $mGp$, where $m>\frac{n}{2}$, $p$  is a large constant to be determined in the end. We still denote $\left(2- \frac{1+2/n}{1+\theta}\right)$ by $B$ with $\theta=\frac{1}{n}$. Then, we have
\begin{equation}\label{key22}  
\begin{split}
\Delta_{\Phi}\left(|S_{\Phi}|^{m(p-1)}v^{mGp}\right)&\geq mG|S_{\Phi}|^{m(p-1)+2}v^{mGp} + m(p-1)\bar{\nabla}_{i}\tilde{E}_{(12)}^{i} \\
 &+ \tilde{M}|S_{\Phi}|^{m(p-1)-2}v^{mGp} |\bar{\nabla} |S_{\Phi}||^2 \\
  &+\tilde{N}|S_{\Phi}|^{m(p-1)} v^{mGp-2}|\bar{\nabla} v|^2,
\end{split}
\end{equation}
where
\begin{equation}\label{Ma}
    \tilde{M}=m(p-1)\left(m(p-1) -\left(2- (1-\varepsilon(\Phi))\left(\frac{1+2/n}{1+\theta}\right)\right)-\varepsilon(\Phi) -\alpha mGp\right),
\end{equation}
\begin{equation}\label{Na}
    \tilde{N}=mGp\left(mGp+1-\frac{p-1}{Gp}\varepsilon(\Phi) - \alpha^{-1}m(p-1)\right),
\end{equation}
\begin{equation}
\tilde{E}_{(12)}^{i} = |S_{\Phi}|^{m(p-1)-2}v^{mGp}\left(\bar{g}^{kl}T^{(1)}_{jlk}h^{ij}+\bar{g}^{il}T^{(2)}_{ijl}h^{kj}\right).
\end{equation}
Let $\eta$ be a test function with compact support. Multiplying on both sides of \eqref{key2} by $|S_{\Phi}|^{m(p+3)-2}v^{mGp}\eta^{2m(p+1)}$ and taking integral, we obtain
\begin{equation}\label{key23}
\begin{split}
    &\quad\, \int_{\Sigma} |S_{\Phi}|^{2m(p+1)}v^{2mGp}\eta^{2m(p+1)} \\
    &\leq \frac{1}{mG} \int_{\Sigma} |S_{\Phi}|^{m(p+3)-2}v^{mGp}\eta^{2p}\Delta_{\Phi}\left(|S_{\Phi}|^{m(p-1)}v^{mGp}\right) \\
    &-\frac{p-1}{G}\int_{\Sigma}|S_{\Phi}|^{m(p+3)-2}v^{mGp}\eta^{2m(p+1)}\bar{\nabla}_{i}\tilde{E}_{(12)}^{i} \\
    &- \frac{\tilde{M}}{mG}\int_{\Sigma} |S_{\Phi}|^{2m(p+1)-4}v^{2mGp} \eta^{2m(p+1)} |\bar{\nabla} |S_{\Phi}||^2\\ 
  &-\frac{\tilde{N}}{mG}\int_{\Sigma} |S_{\Phi}|^{2m(p+1)-2} v^{2mGp-2}\eta^{2m(p+1)}|\bar{\nabla} v|^2.
\end{split}
\end{equation}
 Denoting $|S_{\Phi}|^{m(p-1)}v^{mGp}$ by $f$, then we have
\begin{equation}
    \begin{split}
        \frac{1}{mG}\int_{\Sigma} |S_{\Phi}|^{4m-2}\eta^{2m(p+1)} f \Delta_{\Phi} f  &= -\frac{(4m-2)}{mG}\int_{\Sigma} |S_{\Phi}|^{4m-3}\eta^{2m(p+1)} f(\bar{\nabla} f, \bar{\nabla} |S_{\Phi}|)_{\Phi}  \\
        & \quad \,- \frac{2(p+1)}{G}\int_{\Sigma} |S_{\Phi}|^{4m-2}\eta^{2m(p+1)-1} f(\bar{\nabla} f, \bar{\nabla} \eta)_{\Phi}  \\
        & \quad\,- \frac{1}{mG}\int_{\Sigma} |S_{\Phi}|^{4m-2}\eta^{2m(p+1)} |\bar{\nabla}f|^2  .
    \end{split}
\end{equation}
Applying Young's inequality in the above with $\beta= \frac{1}{10m(p+1)}$ and  $\gamma= \frac{1}{10m+2}$, we have
\begin{equation}\label{key24}
    \begin{split}
         &\qquad\frac{1}{mG}\int_{\Sigma} |S_{\Phi}|^{4m-2}\eta^{2m(p+1)} f \Delta_{\Phi} f\\
         &\leq  - \frac{1}{mG}\left(1 - 2m(p+1)\beta - \frac{(4m-2)}{2}\gamma\right) \int_{\Sigma} |S_{\Phi}|^{4m-2}\eta^{2m(p+1)}|\bar{\nabla}f|^2   \\  
         & + \frac{(4m-2)}{2mG}\gamma^{-1} \int_{\Sigma}|S_{\Phi}|^{4m-4}\eta^{2m(p+1)}f^2|\bar{\nabla}|S_{\Phi}||^2 \\
         & + \frac{(p+1)}{G}\beta^{-1} \int_{\Sigma}|S_{\Phi}|^{4m-2}\eta^{2m(p+1)-2}f^2|\bar{\nabla}\eta|^2    \\
         & \leq  \frac{(4m-2)}{2mG}\gamma^{-1} \int_{\Sigma}|S_{\Phi}|^{4m-4}\eta^{2m(p+1)}f^2|\bar{\nabla}|S_{\Phi}||^2  \\
         & + \frac{p+1}{G}\beta^{-1} \int_{\Sigma}|S_{\Phi}|^{2m(p+1)-2}v^{2mGp}\eta^{2m(p+1)-2}|\bar{\nabla}\eta|^2.
    \end{split}
\end{equation}

Combining (\ref{key23}) and (\ref{key24}), we obtain 
\begin{equation}
    \begin{split}
   & \quad\,\int_{\Sigma} |S_{\Phi}|^{2m(p+1)}v^{2mGp}\eta^{2m(p+1)} \\
   &\leq \,\frac{(p+1)}{G}\beta^{-1} \int_{\Sigma} |S_{\Phi}|^{2m(p+1)-2}v^{2mGp}\eta^{2m(p+1)-2}|\bar{\nabla} \eta|^2 \\
    &- \,\frac{p-1}{G}\int_{\Sigma}|S_{\Phi}|^{m(p+3)-2}v^{mGp}\eta^{2m(p+1)}\bar{\nabla}_{i}\tilde{E}_{(12)}^{i} \\
  &- \frac{\tilde{M}}{mG}\int_{\Sigma} |S_{\Phi}|^{2m(p+1)-4}v^{2mGp} \eta^{2m(p+1)} |\bar{\nabla} |S_{\Phi}||^2\\ 
  &-\frac{\tilde{N}}{mG}\int_{\Sigma} |S_{\Phi}|^{2m(p+1)-2} v^{2mGp-2}\eta^{2m(p+1)}|\bar{\nabla} v|^2\\
  & +  \frac{(4m-2)}{2mG}\gamma^{-1} \int_{\Sigma}|S_{\Phi}|^{4m-4}\eta^{2m(p+1)}f^2|\bar{\nabla}|S_{\Phi}||^2.
\end{split}
\end{equation}

On the other hand, by \eqref{E12prop} we have 
\begin{equation}
|\tilde{E}_{(12)}^i|\leq \varepsilon(\Phi)|S_{\Phi}|^{m(p-1)+1}v^{mGp}.
\end{equation}
Using divergence theorem and Young's inequality, we obtain
\begin{equation}
\begin{split}
&\quad-\frac{p-1}{G}\int_{\Sigma}|S_{\Phi}|^{m(p+3)-2}v^{mGp}\eta^{2m(p+1)}\bar{\nabla}_{i}\tilde{E}_{(12)}^{i} \\
&=\frac{p-1}{G}\int_{\Sigma}\tilde{E}_{(12)}^{i}\bar{\nabla}_{i}\left(|S_{\Phi}|^{m(p+3)-2}v^{mGp}\eta^{2m(p+1)}\right) \\
&\leq \varepsilon(\Phi)\frac{p-1}{G}\biggl(\left(m(p+3)-2\right)\int_{\Sigma} |S_{\Phi}|^{2m(p+1)-2}v^{2mGp}\eta^{2m(p+1)}|\bar{\nabla}|S_{\Phi}||  \\
&+ mGp\int_{\Sigma} |S_{\Phi}|^{2m(p+1)-1}v^{2mGp-1}\eta^{2m(p+1)}|\bar{\nabla}v| \\
& + 2m(p+1)\int_{\Sigma} |S_{\Phi}|^{2m(p+1)-1}v^{2mGp}\eta^{2m(p+1)-1}|\bar{\nabla}\eta|\biggr)\\
&\leq \varepsilon(\Phi)\frac{p-1}{G} \biggl( 3m(p+3)\int_{\Sigma} |S_{\Phi}|^{2m(p+1)}v^{2mGp}\eta^{2m(p+1)}  \\
& + m(p+1)\int_{\Sigma} |S_{\Phi}|^{2m(p+1)-2}v^{2mGp}\eta^{2m(p+1)-2}|\bar{\nabla} \eta|^2  \\
&+\frac{1}{2}\left(m(p+3)-2\right)\int_{\Sigma} |S_{\Phi}|^{2m(p+1)-4}v^{2mGp}\eta^{2m(p+1)}|\bar{\nabla} |S_{\Phi}||^2 \\
&+ \frac{1}{2}mGp\int_{\Sigma} |S_{\Phi}|^{2m(p+1)-2}v^{2mGp-2}\eta^{2m(p+1)}|\bar{\nabla} v|^2
\biggr).
\end{split}
\end{equation}
Combining above we obtain
\begin{equation}\label{I2}
    \begin{split}
    &\quad\int_{\Sigma} |S_{\Phi}|^{2m(p+1)}v^{2mGp}\eta^{2m(p+1)} \\
    &\leq \left(\frac{(p+1)}{G}\beta^{-1}+\frac{2m(p^2-1)}{G}\varepsilon(\Phi)\right) \int_{\Sigma} |S_{\Phi}|^{2m(p+1)-2}v^{2mGp}\eta^{2m(p+1)-2}|\bar{\nabla} \eta|^2 \\
  &+\biggl(\frac{(p-1)\left(m(p+3)-2\right)}{2G}\varepsilon(\Phi)+ \frac{(4m-2)}{2mG}\gamma^{-1}- \frac{\tilde{M}}{mG}\biggr)\\
  &\times\int_{\Sigma} |S_{\Phi}|^{2m(p+1)-4}v^{2mGp}\eta^{2m(p+1)}|\bar{\nabla} |S_{\Phi}||^2\\ 
  &+\left(\frac{(p-1)mp}{2}\varepsilon(\Phi)- \frac{\tilde{N}}{mG}\right)\int_{\Sigma} |S_{\Phi}|^{2m(p+1)-2}v^{2mGp-2}\eta^{2m(p+1)}|\bar{\nabla} v|^2.
    \end{split}
\end{equation}
At this point, we will carefully choose parameters $p$ and make $\varepsilon(\Phi)$ small enough, so that the last two terms of \eqref{I2} are non-positive. Firstly, we fix $m>\frac{n}{2}$, and recall $\theta=\frac{1}{n}$, so we can again have
\begin{equation}
 0<B = \left(2-\frac{1+2/n}{1+\theta}\right)<1.
\end{equation}\\
Then we make $\|\Phi-1\|_{C^{3}(\mathbb{S}^n)}$ very small (depending on $p, m, n$),  so $\|G-1\|$ will be small enough and
\begin{equation}
   0<1-G B<1,
\end{equation}
 and for each  $p$ large enough (to be fixed later), there exists a constant $\alpha$ very close to the left end point of the following interval (see \eqref{interval3} with $p$ replaced by $m(p-1)$ and $q$ replaced by $mGp$) and whenever $\alpha$ lies in following interval, both $\tilde{M}$ and $\tilde{N}$ would be positive,
\begin{equation}\label{interval2}
  \left(\frac{m(p-1)}{mGp+1},\, \frac{m(p-1)-\left(2- \frac{1+2/n}{1+\theta}\right)}{mGp}\right).
\end{equation}
 Then we plug such choice of $\alpha$ into \eqref{Ma} and notice $\|\Phi-1\|_{C^{3}(\mathbb{S}^n)}$  (depending  on $p, m, n$) and $\varepsilon(\Phi)>0$ small enough, so we have
\begin{equation}\label{I3}
    \frac{\tilde{M}}{2mG} \geq \frac{mp(p-1)}{3(mGp+1)}\left(\left(1-GB\right)-\frac{1}{p}-\frac{B}{mp}\right)-\frac{1}{3m} >0,
\end{equation}
and
\begin{equation}
   \tilde{N}>0 .
\end{equation}
We now make $\|\Phi-1\|_{C^{3}(\mathbb{S}^n)}$  further small (depending  on $p, m, n$), so $\varepsilon(\Phi)>0$ will be small enough
such that
\begin{equation}\label{tildeM absorb}
    \frac{(p-1)\left(m(p+3)-2\right)}{2G}\varepsilon(\Phi)- \frac{\tilde{M}}{mG}\leq -\frac{\tilde{M}}{2mG},
\end{equation}
and
\begin{equation}
   \frac{(p-1)mp}{2}\varepsilon(\Phi)- \frac{\tilde{N}}{mG}\leq -\frac{\tilde{N}}{2mG}.
\end{equation}
Now using \eqref{I3}, \eqref{tildeM absorb} and choosing $p$ large enough (depending on $m>\frac{n}{2}$, e.g $p=1000(mn+1)$), we have
\begin{equation}\label{dependence of ephi}
\begin{split}
        &\quad\frac{(p-1)\left(m(p+3)-2\right)}{2G}\varepsilon(\Phi)+ \frac{(4m-2)}{2mG}\gamma^{-1}- \frac{\tilde{M}}{mG}\\
        &\leq \frac{(4m-2)}{2mG}(10m+2)-\frac{(1-GB)p}{4} \\
        &\leq  \frac{(4m-2)}{2mG}\gamma^{-1}- \frac{\tilde{M}}{2mG}\\
        &\leq \frac{(4m-2)}{2mG}(10m+2)-\frac{p}{4(n+1)}\\
        &\leq 0 .
\end{split}
\end{equation}
Hence, \eqref{I2} becomes
\begin{equation}
\begin{split}
  &\quad\,  \int_{\Sigma} |S_{\Phi}|^{2m(p+1)}v^{2mGp}\eta^{2m(p+1)} \\
  &\leq 
 \left(\frac{(p+1)}{G}\beta^{-1}+\frac{2m(p^2-1)}{G}\varepsilon(\Phi)\right) \int_{\Sigma} |S_{\Phi}|^{2m(p-1)}v^{2mGp}\eta^{2m(p-1)}|\bar{\nabla} \eta|^2  .
 \end{split}
\end{equation}
This together with Young's inequality implies that
\begin{equation}
\int_{\Sigma} |S_{\Phi}|^{2m(p+1)}v^{2mGp}\eta^{2m(p+1)}  \leq C_2(p, m, n, \Phi)\int_{\Sigma} v^{2mGp}|\bar{\nabla}\eta|^{2m(p+1)} .
\end{equation}

\end{proof}


\section{Proof of main theorem}

In this section, we put everything from previous two sections together to prove the Theorem \ref{Main Theorem}.
\begin{proof}[{\bf Proof of Theorem \ref{Main Theorem}:}]
For  $m=n>\frac{n}{2}$,  $p$ large enough (depending on $m, n$, e.g $p=1000(mn+1)$), by  Corollary \ref{aMVP}, Proposition \ref{int1} and Proposition \ref{int2} there is some constant $\delta(n)>0$ such that if $\|\Phi-1\|_{C^{3}(\mathbb{S}^n)}\leq \delta(n)$ 
then the two integral estimates in Proposition \ref{int1}, Proposition \ref{int2} hold.  By suitably choosing cut-off function $\eta$ on $\Sigma\cap B_{R}$ with $|\bar{\nabla}\eta|\leq 1/R$ on $\Sigma$, we obtain
    \begin{equation}
        R^{-\frac{n}{2}}\left(\int_{\Sigma\cap B_R} |S_{\Phi}|^{2p}v^{2Gp}\eta^{2p} d\mu_{\Sigma}\right)^{\frac{1}{2}}\leq \tilde{C}_1(n, \Phi)R^{-p}\mathop{\mathrm{sup}}\limits_{\Sigma\cap B_{2R}}v^{Gp}
    \end{equation}
    and
    \begin{equation}
    \begin{split}
        &\quad\, R^{\frac{1}{2}} \left(\int_{\Sigma\cap B_R} |S_{\Phi}|^{2n(p+1)}v^{2nGp}\eta^{2n(p+1)} d\mu_{\Sigma}\right)^{\frac{1}{2n}}  \\
       & \leq \tilde{C}_2(n, \Phi)R^{-p}\mathop{\mathrm{sup}}\limits_{\Sigma\cap B_{2R}}v^{Gp},
    \end{split}
    \end{equation}
where we also used $\mathcal{A}_{\Phi}(\Sigma\cap B_{R})\leq \tilde{C}(n, \Phi)R^n$ since anisotropic minimal graphs are area minimizers compared with the spherical competitors with radius $R$.

The above two inequalities together with \eqref{MVP_equality} in Corollary \ref{aMVP} and the gradient growth condition \eqref{gradient_growth} imply
\begin{equation}\label{last}
\begin{split}
         |S_{\Phi}|v^{G}(0)&\leq C(n, \Phi) R^{-1}\mathop{\mathrm{sup}}\limits_{\Sigma\cap B_{2R}}v^{G}\\
         &\leq C( n, \Phi) o(1) R^{G\varepsilon-1}.
\end{split}
\end{equation}
Then, we set $\varepsilon_0(n, \Phi)=G^{-1}$, where $G=G(n, \Phi)$ is defined in \eqref{G}, so $\varepsilon_0(n, \Phi)\to 1$ as  $\|\Phi-1\|_{C^{3}(\mathbb{S}^n)}\to 0$. Let $R\to \infty$, we have for every $\varepsilon\leq \varepsilon_0(n, \Phi)$ and $|\nabla u|$  satisfying \eqref{gradient_growth}, the second fundamental form vanishes everywhere and the corresponding smooth entire anisotropic minimal graph of $u$ must be a hyperplane. This completes the proof of Theorem \ref{Main Theorem}.
\end{proof}
\begin{rem}
We remark that by the work of Lin \cite{Lin} and the work of Colding and Minicozzi \cite{CM}, Bernstein theorem and related curvature estimates hold in dimension $n = 2$ under the hypothesis of only $C^2$-closeness between anisotropic integrand and the classical area integrand. It will be an interesting problem if our main result Theorem \ref{Main Theorem} still holds under  assuming only weaker integrand closeness condition (e.g $C^2$-closeness condition) between anisotropic integrand and the classical area integrand.\\ 
\end{rem}


\begin{rem}\label{nonlinear}
  There are nonlinear solutions for area functional case with gradient growth condition $|Du(x)| = O(|x|^{1+\frac{1}{n}})$ by the work of Simon \cite{Si}, and the lowest eigenvalue of the Jacobi operator of an area minimizing cone gives a lower bound on gradient growth of the above form (see \cite[Sec 5, Sec 6]{Si}). It would also be very interesting to consider whether the gradient growth condition in our main result Theorem \ref{Main Theorem} can be relaxed to $|Du(x)| = o(|x|)$  in the case of $C^3$-closeness to the area as in the result in \cite{EH}. This question is subtle and our perturbation argument may not be robust enough to justify this question.  However, to justify the question, one might be able to try some compactness argument  which leads to a contradiction to the result of Ecker-Huisken \cite{EH}.
\end{rem}



\end{document}